\documentclass[12pt,twoside]{article}
\usepackage[paperheight=30cm, paperwidth=22cm,
textwidth=17cm,
headheight=1mm,
]{geometry}
\usepackage[utf8]{inputenc}
\usepackage{amsmath}
\usepackage{titlesec}
\titleformat*{\section}{\normalfont\fontfamily{phv}\fontsize{15}{17}\bfseries}
\titleformat*{\subsection}{\normalfont\fontfamily{phv}\fontsize{13}{17}\bfseries}
\titlespacing*{\subsection}{0pt}{0\baselineskip}{0\baselineskip}
\usepackage{amsthm}
\usepackage{chngcntr}
\counterwithin{figure}{section}
\usepackage{amsfonts}
\usepackage{bbm}
\usepackage{tikz-cd}
\usepackage{mathtools}
\usepackage{amssymb}
\usepackage{mathrsfs}
\usepackage[english]{babel}
\usepackage{changepage}
\usepackage{graphicx}
\graphicspath{ {./images/} }
\usepackage[onehalfspacing]{setspace}
\usepackage{listings}
\usepackage{fancyhdr}
\usepackage{layout}

\let\tmp\oddsidemargin
\let\oddsidemargin\evensidemargin
\let\evensidemargin\tmp
\reversemarginpar

\newtheorem{theorem}{Theorem}[section]

\newtheorem{remark}{Remark}[section]

\newtheorem{lemma}[theorem]{Lemma}
\newtheorem{proposition}[theorem]{Proposition}

\usepackage{hyperref}
\hypersetup{
	colorlinks=true,
	linkcolor = blue,
	filecolor = blue,
	citecolor = blue,
	urlcolor = blue
}

\title{Topological Sequence Entropy of co-Induced Systems}
\author{Dakota M. Leonard}
\date{} 
\pagebreak
\begin{document}
	\flushbottom
	\raggedbottom
	\maketitle
	\begin{abstract}\normalfont\small\singlespacing
		\noindent Let $G$ be a discrete, countably infinite group and $H$ a subgroup of $G$. If $H$ acts continuously on a compact metric space $X$, then we can induce a continuous action of $G$ on $\prod_{H\backslash G}X,$ where $H\backslash G$ is the collection of right-cosets of $H$ in $G$. This process is known as the co-induction. In this article, we will calculate the maximal pattern entropy of the co-induction. If $[G:H] < +\infty$ we will show that the $H$ action is null if and only if the co-induced action of $G$ is null. Additionally, we will discuss an example where $H$ is a proper subgroup of $G$ with finite index, and we will show that the maximal pattern entropy of the $H$-action on $X$ is equal to the maximal pattern entropy of the co-induced action of $G$ on $\prod_{H \backslash G} X$. If $[G:H] = +\infty,$ we will show that the maximal pattern entropy of the co-induction is always $+\infty$ given the $H$-system is not trivial.  
	\end{abstract}
	\section{Introduction}
	By a \textbf{topological dynamical system} (t.d.s.) we mean a tuple $(X, G,\alpha)$ where $X$ is a compact metric space, $G$ is a group, and $\alpha$ is an action $(g,x)\mapsto \alpha_{g}(x)$ of $G$ on $X$ by homeomorphisms. Unless stated otherwise, throughout this document all groups will be assumed to be infinite, countable, discrete groups with identity $e$. In 1965, Adler et al. introduced an isomorphism invariant for t.d.s. known as topological entropy to distinguish $\mathbb{Z}$-systems. Later, in 1985, Ollagnier proposed an extension of this concept to actions of amenable groups. Ornstein and Weiss further developed this extension in 1987. More recently, in 2011, Kerr and Li extended the concept of topological entropy to sofic groups (see \cite{adler_konheim_mcandrew_topological_entropy},\cite{Kerr_Li_sofic_entropy},\cite{Kerr_Li_sofic_amenable_entropy},\cite{Ollagnier_book}, \cite{Ornstein_Weiss_amenable_entropy},\cite{Walters_book}). If the topological entropy of a t.d.s. is 0, we call the system \textbf{deterministic}. One way to distinguish between deterministic systems was introduced by Goodman in \cite{Goodman_sequence_entropy} for $\mathbb{Z}$-systems known as topological sequence entropy. The definition proposed by Goodman can be extended to any group $G$ giving us an isomorphism invariant that does not depend on amenability or soficity. Huang and Ye in \cite{Huang_Ye_maximal_sequence_entropy} improved upon this definition with a new invariant \textbf{maximal pattern entropy} defined for any $G$-system. In particular, they showed the maximal pattern entropy of a t.d.s. is equal to the supremum of all topological sequence entropy over all sequences of $G.$ Huang and Ye proved another equivalent definition for maximal pattern entropy using local entropy theory and $IN$-tuples. \\
	\indent The study of local entropy theory started in 1993 by Blanchard. Blanchard introduced entropy pairs to characterize uniformly positive entropy (u.p.e.), completely positive entropy (c.p.e.), and relate topological entropy to topological mixing properties. Particularly, he proved that u.p.e. implies weak-mixing for $\mathbb{Z}$-systems (see \cite{Blanchard_fully_positive_topological_entropy},\cite{Blanchard_entropy_pairs_disjointness_theorem}). Huang et al. in \cite{Huang_Shao_Ye_sequence_entropy_pairs} defined sequence entropy pairs and uniformly positive sequence entropy (u.p.s.e.) to describe when a topological dynamical system is weak-mixing. Improving upon the result of Blanchard, they proved weak-mixing is equivalent to u.p.s.e. for $\mathbb{Z}$-systems (see \cite{Huang_Shao_Ye_sequence_entropy_pairs},\cite{Huang_Maass_Ye_sequence_entropy_tuples},\cite{Huang_Shao_Ye_mixing_sequence_entropy} for more information about sequence entropy pairs and tuples). Both proofs of these results can be easily extended to any abelian group. Kerr and Li in \cite{Kerr_Li_Independence} studied combinatorial independence and its relation to entropy, sequence entropy, and tameness. They defined and characterized the properties of $IE$-tuples, $IN$-tuples, and $IT$-tuples (independence entropy, null, and tame tuples, respectively). They showed the existence of a non-diagonal $IN$-pair implies positive topological sequence entropy for some sequence of group elements. Moreover, they showed any non-diagonal $IN$-tuple is the same thing as a sequence entropy tuple. Huang and Ye in \cite{Huang_Ye_maximal_sequence_entropy} proved that maximal pattern entropy can be calculated by determining the existence of $IN$-tuples. Particularly, they showed that maximal pattern entropy is $+\infty$ or $\log(k)$ where $k$ is the longest length of an $IN$-tuple of distinct points that exists in the system.\\
	\indent In this article, we calculate the maximal pattern entropy of $G$-actions induced from actions of subgroups by the co-induction construction using local entropy theory and $IN$-tuples (see Section \ref{co-induction prelim section} for details about the co-induction construction). In doing so, we examine the interplay between the dynamical and algebraic properties of groups. We will often refer to these t.d.s. as \textbf{co-induced systems}. The theory of co-induced systems has been used to study and answer questions in dynamical systems. Stepin in \cite{Stepin_bernoulli_shifts_on_groups} used co-induced systems to study Ornstein groups. Specifically, he proved that groups which contain an Ornstein subgroup are Ornstein groups (see \cite{Lewis_Bowen_weak_isomorphism}). Dooley et al. in \cite{Dooley_Golodets_Rudolph_Sinel_co_induction} used co-induced systems to create non-Bernoulli systems that have measure-theoretic c.p.e. for every amenable group $G$ that contains an element of infinite order. They proved if $H\leq G$ and $G$ amenable then the measure-theoretic entropy is preserved by the co-induction. Dooley and Zhang in \cite{Dooley_Zhang_co_induction} studied more properties of co-induced systems of amenable groups. They showed entropy properties including topological entropy, c.p.e., and u.p.e. are preserved by the co-induction. Furthermore, they investigated topological dynamical properties such as strong mixing, weak-mixing, and topological transitivity for co-induced systems. Hayes in \cite{Ben_hayes_co_induction_sofic_entropy} showed if $G$ is a sofic group and $H\leq G$ then the topological sofic entropy is preserved by the co-induction. It is unknown if measure-theoretic sofic entropy is preserved by the co-induction. \\
	\indent Let $G$ be a group and $H\leq G.$ Our study will be split into two parts. In Section \ref{maximal pattern entropy finite index subsection} the maximal pattern entropy of co-induced systems where $[G:H] < +\infty$ is analyzed. The first result in this section is a proposition characterizing $IN$-tuples of $G$-systems under the condition that $G$ contains a finite index subgroup. 
	\addtocounter{section}{2}
	\begin{proposition}
		Let $(X,G,\alpha)$ be a t.d.s. and $H\leq G$ such that $[G:H] < + \infty,$ then 
		\begin{align*}
			IN_{k}(X,G) = IN_{k}(X,H).
		\end{align*}
	\end{proposition}
	\noindent This proposition will play a crucial role in analyzing the structure of $IN$-tuples of co-induced systems. It serves as key fact when proving the subsequent lemma.\pagebreak
	\begin{lemma}
		Let $H$ be a subgroup of $G$ such that $[G:H] < + \infty$. If $(X,H,\alpha)$ is a t.d.s. and $(X_{H}^{G},G,\alpha^{H\backslash G})$ is the co-induced system, then 
		\begin{align*}
			IN_{k}(X_{H}^{G},G)\subseteq \prod_{\theta\in H\backslash G}IN_{k}(X,H).
		\end{align*}
	\end{lemma}
	\noindent The first major result of this article establishes that a t.d.s. is null if and only if the co-induced system is null. 
	\begin{theorem}
		Let $H$ be a subgroup of $G$, $(X,H,\alpha)$ be a t.d.s. and $(X_{H}^{G},G,\alpha^{H\backslash G})$ be the co-induced system. If $[G:H] < +\infty$ then $(X_{H}^{G},G,\alpha^{H\backslash G})$ is null if and only if $(X,H,\alpha)$ is null. 
	\end{theorem}
	\noindent Following the proof of Theorem \ref{null iff null in finite index case}, we present an example illustrating that the finite index condition is essential for the preceding results to hold.\\
	\indent When $[G:H] < +\infty,$ it is natural to ask whether the maximal pattern entropy of the co-induction equals that of the product action. The next major result of this article confirms this to be true, provided the subgroup $H$ is contained within the center of $G.$ 
	\begin{theorem}
		Let $H\leq G$ such that $H\subseteq Z(G)$ and $[G:H] < +\infty.$ If $(X,H,\alpha)$ is a t.d.s. and $(X_{H}^{G},G,\alpha^{H\backslash G})$ is the co-induced system, then $$h_{top}^{*}(X_{H}^{G},G) = [G:H]h_{top}^{*}(X,H).$$ 
	\end{theorem}
	\noindent Following the proof of Theorem \ref{center subgroup theorem}, we conclude Section \ref{maximal pattern entropy finite index subsection} by introducing a t.d.s. $(X_1,\mathbb{Z})$ developed by Tan, Ye, and Zhang in \cite{TAN_Ye_Zhang_set_of_sequence_entropes} where $h_{top}^{*}(X_1,\mathbb{Z}) = \log(2).$ Let $G = \mathbb{Z}\rtimes_{\psi} \mathbb{Z}/2\mathbb{Z}$ where $\psi: \mathbb{Z}/2\mathbb{Z} \to Aut(\mathbb{Z})$ is the homomorphism defined by $\psi(x)(n) = -n$ and $\langle x\rangle = \mathbb{Z}/2\mathbb{Z}.$ Following the ideas of Snoha, Ye, and Zhang in \cite{Snhoa_Ye_Zhang_topological_sequence_entropy}, we show that $h_{top}^{*}((X_1)_{\mathbb{Z}}^{G},G) = \log(2)$ (Theorem \ref{co-induction counter-example calculation}). This result demonstrates the necessity of $H\subseteq Z(G)$ in Theorem \ref{center subgroup theorem}.\\
	\indent In Section \ref{maximal pattern entropy infinite index subsection}, we analyze the maximal pattern entropy of co-induced systems where $[G:H] = +\infty$. Using a group theory result by Neumann from \cite{Neumann_group_lemma}, we prove the final major result of this article. 
	\addtocounter{theorem}{9}
	\begin{theorem}
			Let $(X,H,\alpha)$ be a t.d.s. where $X$ contains at least two distinct points and $(X_{H}^{G},G,\alpha^{H\backslash G})$ be the co-induced system. If $[G:H] = +\infty$, then $$IN_{k}\left(X_{H}^{G},G\right) = \left(X_{H}^{G}\right)^{k}$$ for all $k\in \mathbb{N}$ and $h_{top}^{*}(X_{H}^{G},G) = +\infty.$ 
	\end{theorem}
	\addtocounter{section}{-2}
	\addtocounter{theorem}{1}
	\section{Preliminaries}
	
	In this section we will go through some notation we'll use throughout the paper. It will be split into two sections. In Section \ref{sequence entropy prelim} we will review definitions and results about topological sequence entropy, maximal pattern entropy, and $IN$-tuples. In Section \ref{co-induction prelim section} we will review the co-induction construction, and some known results involving the co-induction. For any set $X$, $\mathscr{F}(X)$ will represent the set of non-empty finite subsets of $X$, and we write for the diagonal $\Delta_k(X) = \{(x,\dots,x):x\in X\}$ in $X^k$. For $n\in \mathbb{N}$, we will denote $\{1,\dots,n\} = [n]$. A continuous map $\varphi: (X,G,\alpha) \to (Y,G,\beta)$ between two t.d.s. is called a \textbf{$G$-factor map} if $\varphi$ is surjective and $\varphi(\alpha_{g}(x)) = \beta_{g}(\varphi(x))$ for every $x\in X, g\in G$. We say that $(X,G,\alpha)$ and $(Y,G,\beta)$ are \textbf{conjugate} denoted by $(X,G,\alpha)\cong (Y,G,\beta)$ if $\varphi$ is a homeomorphism and a $G$-factor map. We will denote the set of open coverings of $X$ by $\mathcal{C}_{X}^{0}.$ 
	\subsection{Sequence Entropy, Maximal Pattern Entropy, and IN-Tuples}\label{sequence entropy prelim} Goodman in \cite{Goodman_sequence_entropy} formulated topological sequence entropy for integer actions on compact metric spaces. In a similar fashion, we can define topological sequence entropy for arbitrary group actions. Let $(X,G,\alpha)$ be a t.d.s. and $\mathcal{S}(G)$ represent the set of all sequences of $G$. For $\mathfrak{s}\in \mathcal{S}(G)$, we define the \textbf{topological sequence entropy of $(X,G,\alpha)$ with respect to $\mathfrak{s} = \{g_i\}_{i\in \mathbb{N}}$ and $\mathcal{U}\in \mathcal{C}_X^0$} by 
	\begin{align*}
		h_{top}^{\mathfrak{s}}(\mathcal{U},G) = \limsup_{n\to\infty}\frac{1}{n}\log\left(N\left(\bigvee_{i=1}^{n}\alpha_{g_i^{-1}}(\mathcal{U})\right)\right).
	\end{align*}
	where $N(\mathcal{U})$ is the minimal cardinality among all cardinalities of the sub-coverings of $\mathcal{U}.$ The \textbf{topological sequence entropy of $G$ acting on $X$ with respect to $\mathfrak{s}$} is 
	\begin{align*}
		h_{top}^{\mathfrak{s}}(X,G,\alpha) = \sup_{\mathcal{U}\in \mathcal{C}_{X}^{0}}h_{top}^{\mathfrak{s}}(\mathcal{U},G).
	\end{align*}
	If the action is assumed, we will write $h_{top}^{\mathfrak{s}}(X,G).$ \\
	\indent Huang and Ye in \cite{Huang_Ye_maximal_sequence_entropy} introduced \textbf{maximal pattern entropy} for arbitrary group actions. We define 
	\begin{align*}
		p_{X,G,\mathcal{U}}^{*}(n) = \max_{(g_1,\dots,g_n)\in G^{n}}N\left(\bigvee_{i=1}^{n}\alpha_{g_{i}^{-1}}(\mathcal{U})\right).
	\end{align*}	
	For $(g_1,\dots,g_{n+m})\in G^{n+m},$ we have that 
	\begin{align*}
		N\left(\bigvee_{i=1}^{n+m}\alpha_{g_{i}^{-1}}(\mathcal{U})\right)\leq N\left(\bigvee_{i=1}^{n}\alpha_{g_{i}^{-1}}(\mathcal{U})\right)	N\left(\bigvee_{i=n+1}^{n+m}\alpha_{g_{i}^{-1}}(\mathcal{U})\right);
	\end{align*}
	therefore, $p_{X,G,\mathcal{U}}^{*}(n+m)\leq p_{X,G,\mathcal{U}}^{*}(n)p_{X,G,\mathcal{U}}^{*}(m)$ and the sequence $\{\log(p_{X,G,\mathcal{U}}^{*}(n))\}_{n\in\mathbb{N}}$ is sub-additive. By Fekete's Lemma we have that the limit of the sequence $\left\{\frac{\log(p_{X,G,\mathcal{U}}^{*}(n))}{n}\right\}_{n\in\mathbb{N}}$ exists and 
	\begin{align*}
		\lim_{n\to\infty}\frac{1}{n}\log(p_{X,G,\mathcal{U}}^{*}(n)) = \inf_{n\in\mathbb{N}}\frac{1}{n}\log(p_{X,G,\mathcal{U}}^{*}(n)).
	\end{align*} 
	We define the \textbf{maximal pattern entropy with respect to $\mathcal{U}$} by
	\begin{align*}
		h_{top}^{*}(\mathcal{U},G) = \lim_{n\to\infty}\frac{1}{n}\log(p_{X,G,\mathcal{U}}^{*}(n)),
	\end{align*}
	and the \textbf{maximal pattern entropy of $G$ acting on $X$} by
	\begin{align*}
		h_{top}^{*}(X,G,\alpha) = \sup_{\mathcal{U}\in \mathcal{C}_X^0}h_{top}^{*}(\mathcal{U},G).
	\end{align*}
	If the action is assumed, we will write $h_{top}^{*}(X,G)$ to represent the maximal pattern entropy of $G$ acting on $X.$\\ 	
	\indent Huang and Ye proved another way to define maximal pattern entropy involving topological sequence entropy. They first proved it for $\mathbb{Z}$-actions (see Theorem 2.1(a) in \cite{Huang_Ye_maximal_sequence_entropy}), but their proof generalizes for any group actions. 
	\begin{theorem}[Huang and Ye, Theorem A.1 \cite{Huang_Ye_maximal_sequence_entropy}]\label{maximal pattern entropy sup def}
		Let $(X,G,\alpha)$ be a t.d.s. If $\mathcal{U}\in \mathcal{C}_{X}^{0}$, there exists $\mathfrak{s}\in \mathcal{S}(G)$ consisting of distinct elements of $G$ such that 
		$$h_{top}^{*}(\mathcal{U},G) = h_{top}^{\mathfrak{s}}(\mathcal{U},G).$$ Moreover, $$h_{top}^{*}(X,G) = \sup_{\mathfrak{s}\in S(G)}h_{top}^{\mathfrak{s}}(X,G).$$
	\end{theorem}
	Using Theorem \ref{maximal pattern entropy sup def} Huang and Ye prove the following theorem.
	\begin{theorem}[Huang and Ye, Theorem 2.3(2) \cite{Huang_Ye_maximal_sequence_entropy}]\label{finite index subgroup maximal pattern entropy}
		Let $(X,G,\alpha)$ be a t.d.s. and $H\leq G$ such that $[G:H] < + \infty,$ then 
		\begin{align*}
			h_{top}^{*}(X,G,\alpha) = h_{top}^{*}(X,H,\alpha\rvert_{H}).
		\end{align*}
	\end{theorem}
	Huang and Ye originally proved Theorem \ref{finite index subgroup maximal pattern entropy} when $G = \mathbb{Z}$ and $H = k\mathbb{Z}$. The proof they use generalizes for any group $G$ and subgroup $H$ when $[G:H] < +\infty.$ We will outline another proof of Theorem \ref{finite index subgroup maximal pattern entropy} using local entropy theory at the beginning of Section \ref{maximal pattern entropy finite index subsection}. \\
	\indent Using the separated/spanning set definition of sequence entropy presented by Goodman in \cite{Goodman_sequence_entropy}, Goodman proved for any $\mathfrak{s}\in S(G)$ that $$h_{top}^{\mathfrak{s}}(X^k,G,\alpha^{k}) = kh_{top}^{\mathfrak{s}}(X,G,\alpha)$$ where $\alpha^{k}$ is the product action $\alpha_g^{k}(x_1,\dots,x_k) = (\alpha_g(x_1),\dots,\alpha_{g}(x_k))$ (Proposition 2.4 \cite{Goodman_sequence_entropy}). Goodman originally proved this result for $\mathbb{Z}$ actions, and his proof generalizes for any group $G.$ Using Goodman's result and Theorem \ref{maximal pattern entropy sup def}, the following  result holds 
	\begin{proposition}[Huang and Ye, Lemma 2.5 \cite{Huang_Ye_maximal_sequence_entropy}]\label{maximal pattern entropy product formula}
		Let $(X,G,\alpha)$ be a t.d.s., then $$h_{top}^{*}(X^k,G,\alpha^{k}) = kh_{top}^{*}(X,G,\alpha).$$
	\end{proposition}
	The second definition is characterized by local entropy theory and $IN$-tuples. $IN$-tuples were first introduced by Kerr and Li in \cite{Kerr_Li_Independence} where topological sequence entropy was studied in terms of independence. Particularly, the authors used $IN$-tuples to determine whether a t.d.s. is \textbf{null}. A t.d.s. is null if $h_{top}^{\mathfrak{s}}(X,G) = 0$ for every $\mathfrak{s}\in S(G)$, and \textbf{nonnull} if $h_{top}^{\mathfrak{s}}(X,G) > 0$ for some $\mathfrak{s}\in S(G).$ Using Theorem \ref{maximal pattern entropy sup def}, we see that a t.d.s. is null if and only if $h_{top}^{*}(X,G) = 0.$ \\
	\indent We begin our discussion of $IN$-tuples by first defining independence. A set $I\subseteq G$ is said to be an \textbf{independence set} for a tuple $(A_1,\dots,A_k)$ of subsets of $X$ if 
	\begin{align*}
		\bigcap_{g\in F}\alpha_{g^{-1}}(A_{\omega(g)}) \neq \emptyset
	\end{align*}
	for every non-empty finite subset $F$ of $I$ and every function $\omega: F\to [k].$ A tuple $(x_1,\dots,x_k)\in X^{k}$ is said to be an \textbf{$IN$-tuple} if for any product neighborhood $U_1\times\ldots\times U_k$ of $(x_1,\dots,x_k)$ the tuple $(U_1,\dots,U_k)$ has arbitrarily large finite independence sets. We denote the set of IN-tuples of length $k$ by $IN_{k}(X,G,\alpha).$ If the action is assumed we will simply write $IN_k(X,G).$ The following theorem proven by Kerr and Li lists some basic properties of $IN$-tuples. 
	\begin{theorem}[Kerr and Li, Proposition 5.4 \cite{Kerr_Li_Independence}]\label{properties of IN-tuples}
		Let $(X,G,\alpha)$ be a t.d.s. The following are true:\\
		1. Let $(A_1,\dots,A_k)$ be a tuple of closed subsets of $X$ which has arbitrarily large finite independence sets. Then there exists an IN-tuple $(x_1,\dots,x_k)$ with $x_j\in A_j$ for all $1\leq j\leq k.$\\
		2. $IN_2(X,G)\backslash \Delta_2(X)$ is nonempty if and only if $(X,G,\alpha)$ is nonnull. \\
		3. $IN_k(X,G)$ is a closed $G$-invariant subset of $X^k$ under the product action.\\
		4. Let $\pi:(X,G,\alpha)\to (Y,G,\beta)$ be a G-factor map. Then $(\pi\times\ldots\times\pi)(IN_k(X,G)) = IN_k(Y,G).$\\
		5. Suppose that $Z$ is a closed $G$-invariant subset of $X.$ Then $IN_k(Z,G)\subseteq IN_k(X,G).$ 
	\end{theorem}Huang et al in \cite{Huang_Shao_Ye_sequence_entropy_pairs} introduced sequence entropy pairs to characterize u.p.s.e. and null systems. A pair $(x_1,x_2)\in X^{2}\backslash\Delta_2(X)$ is a \textbf{sequence entropy pair} if for any closed, disjoint neighborhoods $A_1,A_2$ of $x_1,x_2$, respectively, there exists a $\mathfrak{s}\in S(G)$ such that $h_{top}^{\mathfrak{s}}(X,G,\{A_1^c,A_2^c\}) > 0.$  A t.d.s. is u.p.s.e. if every pair in $X^2\backslash \Delta_2(X)$ is a sequence entropy pair. Before we state the next result, recall a t.d.s. $(X,G,\alpha)$ is called \textbf{weak-mixing} if for all non-empty open sets $U_1,U_2,V_1,V_2$ there exists a $g\in G$ such that $U_1\cap \alpha_{g}(V_1)\neq \emptyset$ and $U_2\cap \alpha_{g}(V_2)\neq \emptyset.$ 
	\begin{theorem}[Huang, Li, Shao, Ye, Theorem 2.1 \cite{Huang_Shao_Ye_sequence_entropy_pairs}]\label{upse equivalence to weakmixing}
		Let $G$ be an abelian group and $(X,G,\alpha)$ be a t.d.s. Then $(X,G,\alpha)$ is weak-mixing if and only if $(X,G,\alpha)$ is u.p.s.e.
	\end{theorem}
	Huang et al in \cite{Huang_Maass_Ye_sequence_entropy_tuples} expanded the definition of sequence entropy pairs to \textbf{sequence entropy tuples}. A tuple $(x_1,\dots,x_k)\in X^{k}\backslash \Delta_{k}(X)$ is a sequence entropy tuple if whenever $U_1,\dots,U_n$ are pairwise disjoint closed neighborhoods of distinct points in the list $\{x_1,\dots,x_k\}$, there exists $\mathfrak{s}\in S(G)$ such that $h_{top}^{\mathfrak{s}}(\{U_1^c,\dots,U_n^{c}\}) > 0.$ Kerr and Li proved the following result about the relation between sequence entropy tuples and $IN$-tuples. 
	\begin{theorem}[Kerr and Li, Theorem 5.9 \cite{Kerr_Li_Independence}]\label{sequence entropy tuples same as IN-tuples}
		Let $(x_1,\dots,x_k)$ be a tuple in $X^{k}\backslash \Delta_k(X)$ with $k\geq 2.$ Then $(x_1,\dots,x_k)$ is a sequence entropy tuple if and only if it is an $IN$-tuple.
	\end{theorem}
	Huang and Ye in \cite{Huang_Ye_maximal_sequence_entropy} introduced \textbf{maximal pattern entropy tuples}. A tuple $(x_1,\dots,x_k)\in X^{k}\backslash \Delta_{k}(X)$ is called a \textbf{maximal pattern entropy k-tuple}, if whenever $U_1,\dots,U_n$ are pairwise disjoint closed neighborhoods of distinct points in the list $\{x_1,\dots,x_k\}$ we have that $h_{top}^{*}(\{U_1^{c},\dots,U_k^{c}\},G) > 0.$ By Theorem \ref{maximal pattern entropy sup def}, it follows immediately that maximal pattern entropy $k$-tuples are the same thing as sequence entropy $k$-tuples. Moreover, by Theorem \ref{sequence entropy tuples same as IN-tuples} it follows that maximal pattern entropy tuples are the same thing as non-diagonal $IN$-tuples.\\
	\indent Huang and Ye in \cite{Huang_Ye_maximal_sequence_entropy} investigated the relationship between maximal pattern entropy and \textbf{intrinsic $IN$-tuples}. A tuple $(x_1,\dots,x_k)$ is said to be an intrinsic $IN$-tuple if $(x_1,\dots,x_k)\in IN_k(X,G)$ and $x_i\neq x_j$ if $i\neq j.$ We will denote the set of intrinsic $IN$-tuples of length $k$ by $IN_k^{e}(X,G).$ They proved the following theorem that gives us a relation between maximal pattern entropy and intrinsic $IN$-tuples.  
	\begin{theorem}[Huang and Ye, Theorem A.3 \cite{Huang_Ye_maximal_sequence_entropy}]\label{local maximal pattern entropy}
		Let $(X,G,\alpha)$ be a t.d.s., then 
		\begin{align*}
			h_{top}^{*}(X,G) = \log(\sup\{k:IN_{k}^{e}(X,G)\neq \emptyset\}).
		\end{align*}
		Particularly, $h_{top}^{*}(X,G) = +\infty$ or $\log(k)$ for some $k\in \mathbb{N}.$ 
	\end{theorem}  
	
	\subsection{co-Induced Systems from Subgroup Actions }\label{co-induction prelim section} 	Let $(X,H,\alpha)$ be a t.d.s., $\rho$ a metric on $X$ that generates the topology of $X$, and $H\leq G$ where $G$ is a group. Define the space 
	\begin{align*}
		X_{H,\alpha}^{G} := \{f\in X^{G} : f(hg) = \alpha_h(f(g)), \forall h\in H, \forall g\in G\}.
	\end{align*}
	We first need to show that $X_{H,\alpha}^{G}$ is a compact metric space. To do so, we consider $X^G$. Index $G$ by $g_1 = e, g_2,g_3,\dots,$ and define 
	\begin{align*}
		\rho_{G}(f_1,f_2) = \sum_{i=1}^{\infty}\frac{\rho(f_1(g_i),f_2(g_i))}{2^{i}}.
	\end{align*}
	We have that $\rho_{G}$ is well-defined and induces the product topology on $X^{G}.$ Since $X^{G}$ is the Cartesian product of compact spaces, by Tychonoff's Theorem it follows that $X^{G}$ is a compact metric space. The subspace $X_{H,\alpha}^{G}$ is a closed subset of $X^{G};$ therefore, $X_{H,\alpha}^{G}$ is a compact metric space. Let $G$ act on $X^{G}$ by $r_{g}: X^{G}\to X^{G}$ where $r_{g}$ is the right shift action i.e. $r_{g}(f(x)) = f(xg).$ The subspace $X_{H,\alpha}^{G}$ is right shift invariant; thus, the tuple $(X_{H,\alpha}^{G},G,r)$ is a t.d.s.\\
	\indent As $H\leq G$ we can decompose $G$ into a right coset representation with respect to $H$, and let $s:H\backslash G \to G$ be a fixed section of $H\backslash G$ where $s(H) = e.$ We define the space
	\begin{align*}
		X_{H}^{G} = \prod_{\theta\in H\backslash G}X,
	\end{align*}
	which is a compact metric space since $H\backslash G$ is countable. Define 
	\begin{align*}
		\psi_s: &X_{H}^{G} \to X_{H,\alpha}^{G},
	\end{align*}
	by $\psi_s(f)(g) = \alpha_{gs(Hg)^{-1}}f(Hg).$ As $gs(Hg)^{-1}\in H$ and $\psi_{s}(f)\in X_{H,\alpha}^{G}$ our map is well-defined. The function $\psi_{s}$ is bijective and continuous; therefore, $\psi_{s}$ is a homeomorphism since we are working with compact metric spaces. Let $g_0,g\in G$ and $f\in X_{H}^{G},$ it follows that 
	\begin{align*}
		r_{g_0}\psi_{s}(f)(g) &= \psi_{s}(f)(gg_0)\\
		&= \alpha_{gg_{0}s(Hgg_0)^{-1}}f(Hgg_0).
	\end{align*}
	Since $s(Hg)g_0s(Hgg_0)^{-1}\in H$ it follows that 
	\begin{align*}
		r_{g_0}\psi_{s}(f)(g) &= \alpha_{gg_{0}s(Hgg_0)^{-1}}f(Hgg_0)\\
		&= \alpha_{gs(Hg)^{-1}}\alpha_{s(Hg)g_0s(Hgg_0)^{-1}}f(Hgg_0).
	\end{align*}
	Define the continuous action $\alpha^{H\backslash G,s}:G\times X_{H}^{G}\to X_{H}^{G}$ by $(g_0,f)\mapsto \alpha_{g_0}^{H\backslash G,s}(f)$ where $$\alpha_{g_0}^{H\backslash G,s}(f)(Hg) = \alpha_{s(Hg)g_0s(Hgg_0)^{-1}}f(Hgg_0).$$ With this action, we have that $(X_{H}^G,G,\alpha^{H\backslash G,s})$ is a t.d.s. and by the above work, we see that 
	\begin{align*}
		r_{g_0}\psi_{s}(f)(g) &= \alpha_{gs(Hg)^{-1}}\alpha_{g_0}^{H\backslash G,s}(f)(Hg)\\
		&= \psi_{s}(\alpha_{g_0}^{H\backslash G,s}f)(g).
	\end{align*} 
	As $\psi_{s}$ is a homeomorphism, $(X_{H,\alpha}^{G},G,r)\cong (X_{H}^{G},G,\alpha^{H\backslash G,s})$ as t.d.s. for any choice of section $s:H\backslash G \to G$ with $s(H) = e.$ For this reason, our system $(X_{H}^{G},G,\alpha^{H\backslash G,s})$ is independent of our choice of section $s$ and we write it as $(X_{H}^{G},G,\alpha^{H\backslash G}).$ We call the t.d.s. $(X_{H}^{G},G,\alpha^{H\backslash G})$ the \textbf{co-induced system} of $(X,H,\alpha).$ For each $\theta\in H\backslash G,$ define the projections $\pi_{\theta}:X_{H}^{G}\to X$ by $f\mapsto f(\theta)$. The projection $\pi_{H}: X_{H}^{G}\to X$ is important to us because it is an $H$-factor map i.e. $\pi_{H}(\alpha_{h}^{H\backslash G}f) = \alpha_{h}\pi_{H}(f)$ for all $f\in X_{H}^{G}$ and $h\in H.$ \\
	\indent Let $G$ be a group and $X$ a compact metric space. If $H = \{e\},$ then $$X_{\{e\}}^{G} = \prod_{G}X$$ and for $g\in G,$ 
	\begin{align*}
		\alpha_{g}^{\{e\}\backslash G}(f)(h) = f(hg) = r_{g}(f)(h).
	\end{align*}
	Therefore, the co-induced system is the t.d.s. $(X^{G},G,r)$ with the right-shift action $r$. This t.d.s. is called a \textbf{Bernoulli action}. This gives us that co-induced systems are generalizations of Bernoulli actions.  \\
	\indent The following is a result we will use relating our group structures to weak-mixing for co-induced systems.
	\begin{proposition}[Dooley and Zhang, Lemma 7.1(2) \cite{Dooley_Zhang_co_induction}]\label{co-induction weak mixing relation}
		Let $G$ be a group and $H$ a subgroup of $G$ such that $KHK^{-1}\subsetneqq G$ for each $K\in \mathscr{F}(G).$ If $(X,H,\alpha)$ is a t.d.s. and $(X_{H}^{G},G,\alpha^{H\backslash G})$ is the co-induced system, then $(X_{H}^{G},G,\alpha^{H\backslash G})$ is weak-mixing.
	\end{proposition} 
	\section{Maximal Pattern Entropy of co-Induced Systems}\label{maximal pattern entropy section} 
	In this section, we will discuss the maximal pattern entropy of co-induced systems. If $(X,H,\alpha)$ is a t.d.s. and $H\leq G,$ we have the co-induced system $(X_{H}^{G},G,\alpha^{H\backslash G}).$ If $\mathcal{U}$ is an open covering of $X$, then $\pi_{H}^{-1}(\mathcal{U})$ is an open covering of $X_{H}^{G}.$ Since $\pi_{H}$ is an $H$-factor map, we see for $(h_1,\dots,h_n)\in H^{n}$ that
	\begin{align*}
		N\left(\bigvee_{i=1}^{n}\alpha^{H\backslash G}_{h_i^{-1}}(\pi_{H}^{-1}(\mathcal{U}))\right) = N\left(\bigvee_{i=1}^{n}\alpha_{h_i^{-1}}(\mathcal{U})\right).
	\end{align*}
	Therefore,
	\begin{align*}
		p_{X_{H}^{G},G,\pi_{H}^{-1}(\mathcal{U})}^{*}(n) \geq p_{X,H,\mathcal{U}}^{*}(n),
	\end{align*}
	and consequently, it follows that $h_{top}^{*}(X_{H}^{G},G)\geq h_{top}^{*}(X,H).$ 
	\subsection{\large{Maximal Pattern Entropy of co-Induced Systems: $[G:H] < +\infty.$}}\label{maximal pattern entropy finite index subsection}
	We will open this sub-section by proving a result about $IN$-tuples of subgroup actions given that subgroup has finite index. 
	\begin{proposition}\label{finite index subgroup maximal pattern entropy IN-tuple lemma}
		Let $(X,G,\alpha)$ be a t.d.s. and $H\leq G$ such that $[G:H] < + \infty,$ then 
		\begin{align*}
			IN_{k}(X,G) = IN_{k}(X,H).
		\end{align*}
	\end{proposition}
	\begin{proof}
		If $(x_1,\dots,x_k)\in IN_{k}(X,H)$ then $(x_1,\dots,x_k)\in IN_{k}(X,G),$ trivially. For the reverse direction, let $(x_1,\dots,x_k)\in IN_{k}(X,G)$ and $U_1\times\ldots\times U_k$ be a product neighborhood of $(x_1,\dots,x_k).$ Since $(x_1,\dots,x_k)$ is an $IN$-tuple, for any $n\in \mathbb{N}$ there exists $F\in \mathscr{F}(G)$ such that $|F| > n[G:H]$ and $F$ is an independence set for $(U_1,\dots,U_k).$ Let $H\backslash G$ be the set of right cosets of $H,$ $\pi:G\to H\backslash G$ the quotient map, and $s:H\backslash G\to G$ a fixed section of $H\backslash G.$ Define for each $\theta\in \pi(F)$  
		\begin{equation*}
			F_{\theta} = Fs(\theta)^{-1} \cap H,
		\end{equation*}
		then $$F = \bigsqcup_{\theta\in \pi(F)}F_{\theta}s(\theta).$$ Since $\pi(F)\in \mathscr{F}(H\backslash G),$ there exists $\theta_\circ\in \pi(F)$ such that $|F_{\theta_\circ}| = \max_{\theta\in \pi(F)}|F_{\theta}|$ and
		\begin{align*}
			|F| &= \sum_{\theta\in \pi(F)}|F_{\theta}|\\
			&\leq [G:H]|F_{\theta_\circ}|.
		\end{align*} 
		Therefore, $|F_{\theta_\circ}| > n$ and $F_{\theta_\circ}\in \mathscr{F}(H).$ It suffices to show that $F_{\theta_\circ}$ is an independence set for $(U_1,\dots,U_k).$ As $F_{\theta_\circ}s(\theta_\circ)\subseteq F$ by the independence of $F$ we know that $F_{\theta_\circ}s(\theta_\circ)$ is an independence set for $(U_1,\dots,U_k).$ Let $\omega\in [k]^{F_{\theta_\circ}},$ then there exists $\overline{\omega}\in [k]^{F_{\theta_\circ}s(\theta_\circ)}$ such that $\overline{\omega}(hs(\theta_\circ)) = \omega(h)$ for all $h\in F_{\theta_\circ}.$ By independence 
		\begin{align*}
			\emptyset \neq \bigcap_{h\in F_{\theta_\circ}}\alpha_{(hs(\theta_\circ))^{-1}}(U_{\overline{\omega}(hs(\theta_\circ))}) &= \alpha_{s(\theta_\circ)^{-1}}\left( \bigcap_{h\in F_{\theta_\circ}}\alpha_{h^{-1}}(U_{\omega(h)})\right),
		\end{align*}
		and $F_{\theta_\circ}$ is an independence set for $(U_1,\dots,U_k)$. Therefore, $(x_1,\dots,x_k)\in IN_{k}(X,H)$ and $IN_{k}(X,G) = IN_{k}(X,H)$. 
	\end{proof}
	Under the assumptions of Proposition \ref{finite index subgroup maximal pattern entropy IN-tuple lemma}, it follows that $IN_{k}^{e}(X,H) = IN_{k}^{e}(X,G).$ Using this fact and Theorem \ref{local maximal pattern entropy}, we obtain another proof for Theorem \ref{finite index subgroup maximal pattern entropy}. The next result is a lemma involving $IN$-tuples of co-induced systems. 
	\begin{lemma}\label{IN-tuples in the finite index case}
		Let $H$ be a subgroup of $G$ such that $[G:H] < + \infty$. If $(X,H,\alpha)$ is a t.d.s. and $(X_{H}^{G},G,\alpha^{H\backslash G})$ is the co-induced system, then 
		\begin{align*}
			IN_{k}(X_{H}^{G},G)\subseteq \prod_{\theta\in H\backslash G}IN_{k}(X,H).
		\end{align*}
	\end{lemma}
	\begin{proof}
		Let $(f_1,\dots, f_k)\in IN_{k}(X_{H}^{G},G).$ $\pi_{H}: X_{H}^{G}\to X$ is an $H$-factor map, by Proposition \ref{finite index subgroup maximal pattern entropy IN-tuple lemma} and Theorem \ref{properties of IN-tuples}(4)
		\begin{align*}
			(\pi_{H}\times\ldots\times\pi_{H})(IN_{k}(X_{H}^{G},G)) &=(\pi_{H}\times\ldots\times\pi_{H})(IN_{k}(X_{H}^{G},H))\\ 
			&= IN_{k}(X,H).
		\end{align*}
		Therefore, $(f_1(H),\dots,f_k(H))\in IN_{k}(X,H).$ By Theorem \ref{properties of IN-tuples}(3) $IN$-tuples are invariant under the product action of $G$ on $(X_{H}^{G})^k$; therefore, $$\left(\alpha_{s(\theta)}^{H\backslash G}f_1,\dots,\alpha_{s(\theta)}^{H\backslash G}f_k\right)\in IN_{k}(X_H^{G},G)$$ for any $\theta \in H\backslash G.$ Since $\left(\alpha_{s(\theta)}^{H\backslash G}f_{i}\right)(H) = f_{i}(\theta)$, then $(f_1(\theta),\dots,f_{k}(\theta))\in IN_{k}(X,H)$ and $(f_1,\dots,f_k)\in \prod_{\theta\in H\backslash G}IN_{k}(X,H)$ as desired.
	\end{proof} 
	By Lemma \ref{IN-tuples in the finite index case}, we have the following theorem. 
	\begin{theorem}\label{null iff null in finite index case}
		Let $H$ be a subgroup of $G$, $(X,H,\alpha)$ be a t.d.s. and $(X_{H}^{G},G,\alpha^{H\backslash G})$ be the co-induced system. If $[G:H] < +\infty$ then $(X_{H}^{G},G,\alpha^{H\backslash G})$ is null if and only if $(X,H,\alpha)$ is null. 
	\end{theorem}
	\begin{proof}
		We showed at the beginning of Section \ref{maximal pattern entropy section} that $h_{top}^*(X_{H}^{G},G)\geq h_{top}^{*}(X,H)$. Thus, if $h_{top}^{*}(X_{H}^{G},G) = 0$ then $h_{top}^{*}(X,H) = 0$. Conversely, if $h_{top}^{*}(X,H) = 0$ then by Theorem \ref{maximal pattern entropy sup def} and Theorem \ref{local maximal pattern entropy} the set $IN_{2}(X,H)\backslash \Delta_2(X) = \emptyset.$ If $h_{top}^{*}(X_{H}^{G},G) > 0,$ by Theorem \ref{properties of IN-tuples}(2) we have that $IN_{2}(X_{H}^{G},G)\backslash \Delta_{2}(X_H^{G}) \neq \emptyset.$ However, by Lemma \ref{IN-tuples in the finite index case} we have that if $(f_1,f_2)\in IN_{2}(X_{H}^{G},G)\backslash \Delta_2(X_H^{G})$ then $(f_1(\theta),f_2(\theta))\in IN_2(X,H)$ for all $\theta\in H\backslash G.$ As $f_1\neq f_2,$ there exists $\theta\in H\backslash G$ such that $f_1(\theta)\neq f_2(\theta)$ and $(f_1(\theta),f_2(\theta))\in IN_2(X,H)\backslash \Delta_2(X).$ This leads us to a contradiction and we must have that $h_{top}^{*}(X_{H}^G,G) = 0.$ 
	\end{proof}
	The condition of $[G:H] < +\infty$ is necessary for Lemma \ref{IN-tuples in the finite index case} and Theorem \ref{null iff null in finite index case} to be true. To see this, let $(\{0,1\},\mathbb{Z},T)$ be a t.d.s. where $T$ is the trivial action i.e. $T^{n}(x) = x$ for all $n\in \mathbb{Z}$. This t.d.s. is null and by Theorem \ref{maximal pattern entropy sup def} and Theorem \ref{local maximal pattern entropy} it follows that $IN_{k}^{e}(\{0,1\},\mathbb{Z}) = \emptyset$ for $k\geq 2.$ We identify $\mathbb{Z}$ as the subgroup $\mathbb{Z}\times \{0\}$ of $\mathbb{Z}\times \mathbb{Z}$. For all $K\in \mathscr{F}(\mathbb{Z}\times\mathbb{Z}),$ $K\mathbb{Z}K^{-1}\subsetneqq \mathbb{Z}\times \mathbb{Z}$ because $[\mathbb{Z}\times\mathbb{Z}:\mathbb{Z}] = +\infty.$ Thus, by Proposition \ref{co-induction weak mixing relation} the co-induced system $\left(\{0,1\}_{\mathbb{Z}}^{\mathbb{Z}\times\mathbb{Z}},\mathbb{Z}\times\mathbb{Z},\alpha^{\mathbb{Z}\backslash \mathbb{Z}\times\mathbb{Z}}\right)$ is weak-mixing. Since $\mathbb{Z}\times\mathbb{Z}$ is abelian by Theorem \ref{upse equivalence to weakmixing} it follows that 
	\begin{align*}
		IN_2\left(\{0,1\}_{\mathbb{Z}}^{\mathbb{Z}\times\mathbb{Z}},\mathbb{Z}\times\mathbb{Z}\right)\backslash \Delta_2\left(\{0,1\}_{\mathbb{Z}}^{\mathbb{Z}\times\mathbb{Z}}\right) = \left(\{0,1\}_{\mathbb{Z}}^{\mathbb{Z}\times\mathbb{Z}}\right)^{2}\backslash \Delta_2\left(\{0,1\}_{\mathbb{Z}}^{\mathbb{Z}\times\mathbb{Z}}\right)
	\end{align*}
	and $h_{top}^{*}\left(\{0,1\}_{\mathbb{Z}}^{\mathbb{Z}\times\mathbb{Z}},\mathbb{Z}\times\mathbb{Z}\right) > 0$ by Theorem \ref{local maximal pattern entropy}.
	We'll see that whenever $[G:H] = +\infty$ and $X$ is not the trivial system i.e. $X \neq \{\ast\}$, then $h_{top}^{*}(X_{H}^{G},G) = +\infty$ (see Theorem \ref{infinite index subgroup co-induction has infinite maximal pattern entropy}). \\
	\indent When $[G:H] = r < +\infty,$ it is natural to compare the maximal pattern entropy of $(X_{H}^{G},G,\alpha^{H\backslash G})$ to the maximal pattern entropy of $(X^r,H,\alpha^{r})$ where $\alpha^{r}$ is the product action of $H$ on $X^{r}$ i.e. $\alpha^{r}_{h}(x_1,\dots,x_r) = (\alpha_{h}(x_1),\dots,\alpha_{h}(x_r)).$ If $H$ is a normal subgroup of $G$ and $s: H\backslash G \to G$ is a fixed section of $H\backslash G$ with $s(H) = e,$ then for $h\in H$ and $f\in X_{H}^{G}$
	\begin{align}\label{normal subgroups fix cosets}
		\alpha_{h}^{H\backslash G}(f)(\theta) = \alpha_{s(\theta)hs(\theta)^{-1}}f(\theta)
	\end{align}
	for every $\theta\in H\backslash G.$ If $H$ is not normal, we can reduce the problem to the normal core $$N_{H} = \bigcap_{g\in G}g^{-1}Hg$$ of $H$. Since $[G:H] = r,$ it follows that $[G:N_{H}] \leq r!$ by considering $N_{H}$ as the kernel of the action of $G$ acting on $H\backslash G$ by right multiplication. By Theorem \ref{finite index subgroup maximal pattern entropy} it follows that $h_{top}^*(X_{H}^{G},G) = h_{top}^{*}(X_{H}^{G},N_{H})$. For $h\in N_{H}$ and $f\in X_{H}^{G}$ we will have the same behavior as $(\ref{normal subgroups fix cosets})$ since elements of $N_{H}$ fix right cosets of $H$. This is not exactly the product action, but if $H$ is a subgroup contained in the center of $G$, $Z(G)$, then for $h\in H$ and $f\in X_{H}^{G}$ we have that $\alpha_{h}^{H\backslash G}(f) = \alpha_{h}^r(f).$ We have the following theorem. 
	\begin{theorem}\label{center subgroup theorem}
		Let $H\leq G$ such that $H\subseteq Z(G)$ and $[G:H] < +\infty.$ If $(X,H,\alpha)$ is a t.d.s. and $(X_{H}^{G},G,\alpha^{H\backslash G})$ is the co-induced system, then $$h_{top}^{*}(X_{H}^{G},G) = [G:H]h_{top}^{*}(X,H).$$ 
	\end{theorem}
	\begin{proof}
		Let $[G:H] = r.$ Define $\psi: (X_{H}^{G},G,\alpha^{H\backslash G})\to (X^{r},H,\alpha^r)$ by $\psi(f) = f.$ By the discussion proceeding the theorem statement, $\psi$ is a $H$-factor map and $\left(X_{H}^{G},H,\alpha^{H\backslash G}\rvert_{H}\right)\cong (X^{r},H,\alpha^{r}).$ Since maximal pattern entropy is an isomorphism invariant $h_{top}^{*}(X_{H}^{G},H) = h_{top}^{*}(X^r,H).$ By  Theorem \ref{finite index subgroup maximal pattern entropy} and Proposition \ref{maximal pattern entropy product formula} it follows that
		\begin{align*}
			h_{top}^{*}\left(X_{H}^{G},G\right) &= h_{top}^{*}\left(X_{H}^{G},H\right)\\
			&= h_{top}^{*}(X^r,H)\\
			&= rh_{top}^{*}(X,H)
		\end{align*}
		as desired. 
	\end{proof}
	If $H$ is not contained in $Z(G)$, then it is possible for $$h_{top}^{*}(X_{H}^{G},G)\neq [G:H]h_{top}^{*}(X,H).$$ In fact,  there exists a group $G$ with a proper subgroup $H$ with $[G:H] < + \infty,$ and a t.d.s. $(X,H,\alpha)$ such that $h_{top}^{*}(X_{H}^{G},G) = h_{top}^{*}(X,H)$. The rest of this section is dedicated to building such an example.\\ 
	\indent Let $H = \mathbb{Z}$ and $G = \mathbb{Z}\rtimes_{\psi} \mathbb{Z}/2\mathbb{Z}$ where $\psi:\mathbb{Z}/ 2\mathbb{Z}\to \text{Aut}(\mathbb{Z})$ is the homomorphism defined by $\psi(x)(n) = -n$ and $\langle x \rangle = \mathbb{Z}/ 2\mathbb{Z}.$ We identify $\mathbb{Z}$ as the normal subgroup $\mathbb{Z}\times \{0\}$ of $G$. The quotient $\mathbb{Z}\backslash G \cong \mathbb{Z}/ 2\mathbb{Z}$ as groups and $[G:\mathbb{Z}] = 2$. Since $(0,x)(n,0) = (-n,x)$ and $(n,0)(0,x)=(n,x)$, we see that $\mathbb{Z}$ is not contained in $Z(G)$. Let $(X,\mathbb{Z},\alpha)$ be a t.d.s.  In Section \ref{co-induction prelim section}, we discussed that the co-induced system $(X_{\mathbb{Z}}^{G},G,\alpha^{\mathbb{Z}\backslash G})$ is independent of our choice of section of $\mathbb{Z}\backslash G.$ We fix the section $s:\mathbb{Z}\backslash G\to G$ to be $s(\mathbb{Z}(0,0)) = (0,0)$ and $s(\mathbb{Z}(0,x)) = (0,x).$ The co-induction action of $G$ acting on $X\times X$ with this fixed section for $(n,i)\in G$ is
	\begin{align*}
		&\alpha_{(n,0)}^{\mathbb{Z}\backslash G}(f)(0) = \alpha_{(0,0)(n,0)(0,0)}f(0) = \alpha_{(n,0)}f(0),\\
		&\alpha_{(n,0)}^{\mathbb{Z}\backslash G}(f)(x) = \alpha_{(0,x)(n,0)(0,x)}f(x) = \alpha_{(-n,0)}f(x),\\
		&\alpha_{(n,x)}^{\mathbb{Z}\backslash G}(f)(0) = \alpha_{(0,0)(n,x)(0,x)}f(x) = \alpha_{(n,0)}f(x),\\
		&\alpha_{(n,x)}^{\mathbb{Z}\backslash G}(f)(x) = \alpha_{(0,x)(n,x)(0,0)}f(0) = \alpha_{(-n,0)}f(0).\\
	\end{align*}
	Therefore, if $U,V$ are subsets of $X$ we have that 
	\begin{align*}
		\alpha_{(n,0)}^{\mathbb{Z}\backslash G}(U\times V) &= \alpha_{(n,0)}(U) \times \alpha_{(-n,0)}(V),\\
		\alpha_{(n,x)}^{\mathbb{Z}\backslash G}(U\times V) &= \alpha_{(n,0)}(V)\times \alpha_{(-n,0)}(U).
	\end{align*}
	We see that $\alpha^{\mathbb{Z}\backslash G}\rvert_{\mathbb{Z}}$ acts as $\alpha$ in the first coordinate and as $\alpha^{-1}$ in the second coordinate. \\
	\indent Tan, Ye, and Zhang in \cite{TAN_Ye_Zhang_set_of_sequence_entropes} build a t.d.s. $(X_1,\mathbb{Z},T)$ where the maximal pattern entropy of the system is $\log(2)$. For completeness we will include their construction in this paper.  Snoha, Ye, Zhang in \cite{Snhoa_Ye_Zhang_topological_sequence_entropy} build a $\mathbb{N}$-system where the maximal pattern entropy of the system is $\log(2)$. The construction is similar to the construction presented in \cite{TAN_Ye_Zhang_set_of_sequence_entropes}. We will utilize some of the ideas Snoha, Ye, and Zhang presented and apply them to Tan, Ye, and Zhang's system. After we discuss the construction of their system, we will show if we take the co-induced system $\left((X_1)_{\mathbb{Z}}^{G},G,T^{\mathbb{Z}\backslash G}\right)$ of $(X_1,\mathbb{Z},T)$ then $h_{top}^{*}((X_1)_{\mathbb{Z}}^{G},G)  = \log(2).$\\
	\indent Let $A$ be the one-point compactification of $\mathbb{Z}$ embedded into $\mathbb{R}^2$ along the unit circle, and $T_1:A \to A$ by $T_1(a_i) = a_{i+1}$ if $i\neq \infty$ and $T_1(a_{\infty}) = a_{\infty}$ (see figure \ref{our embedding of the integers with infinity point on the unit circle}). We begin the construction of $X_1.$ \\
	\indent For each $k\in \mathbb{N},$ we inductively choose non-negative integers $\{n_0^{k},n_{1}^{k},\dots,n_{k}^{k}\}$ first by letting $n_{0}^{k} = 0$ for all $k$, $n_1^{1} = 9,$ $n_{q+1}^{k} > 100\times n_{q}^{k}$ for $q=1,\dots,k-1,$ and $n_{1}^{k+1} > 100\times n_{k}^{k}.$ By induction, it follows that $n_{q}^{k} > 2k+4$ for all $k\in \mathbb{N}$ and $q=1,\dots,k.$ We choose neighborhoods around the points in $A$ contained in $\mathbb{R}^2$ that do the following:\\
	$(1)$ $U^{k}(a_i)$ is a closed neighborhood of $a_i$ for $i\in \mathbb{Z}$ and $U^{k}(a_{\infty})$ is a closed neighborhood of $a_{\infty}$,\\
	$(2)$ $U^{k}(a_{i})\cap U^{k}(a_{j}) = \emptyset$ for $i\neq j$, $U^{k}(a_i)\cap U^{k}(a_{\infty}) = \emptyset$ for $-n_{k}^{k}\leq i\leq n_{k}^{k},$ and $U^{k}(a_i)\subseteq U^k(a_{\infty})$ for $|i| > n_{k}^{k}.$ \\
	$(3)$ $U^{k+1}(a_{i})\subseteq U^{k}(a_{i})$ for all $i\in \mathbb{Z}\cup \{\infty\}$,\\
	$(4)$ $\text{diam}(U^{k}(a_i)) \leq 1/k$ for all $i\in \mathbb{Z}\cup \{\infty\}.$ 
	\\\begin{figure}[h]
		\centering
		\caption{$(A,\mathbb{Z},T_1)$}\label{our embedding of the integers with infinity point on the unit circle}
		\includegraphics[scale=3]{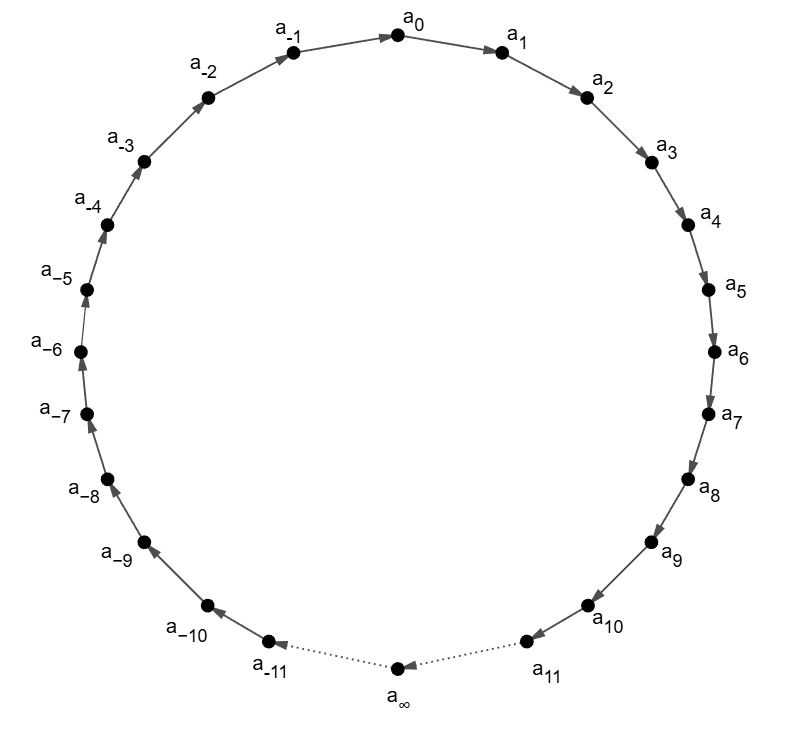}
	\end{figure}
	\indent Let $k\in \mathbb{N}$ and $V_k = \bigcup_{-n_{k}^{k}\leq i \leq n_{k}^{k}} \left(U^k(a_i)\cup U^k(a_{\infty})\right).$ By $(2)$ we have that $A\subseteq V_k$ and $V_{k+1}\subseteq V_{k}.$ By $(3)$ and $(4)$ it follows that $\bigcap_{k=1}^{\infty}V_{k} = A$. 
	For each $k\in \mathbb{N},$ let $F(k) = \{0,1\}^{\{0,\dots,k\}}$ and for $s\in F(k)$ we choose the numbers $j_{s,q}^{k}\in \mathbb{N}$ for $q=1,\dots,k$ such that 
	\begin{align*}
		\left|\frac{n_{q}^k-n_{q-1}^k}{2} - j_{s,q}^{k}\right| \leq \frac{k}{2}
	\end{align*} 
	and define $p_{s,q}^{k} = n_{q}^{k}-n_{q-1}^{k} - j_{s,q}^{k} - s(q).$ It follows that 
	\begin{align*}
		\left|\frac{n_{q}^k-n_{q-1}^k}{2} - p_{s,q}^{k}\right| \leq \frac{k}{2} + 1
	\end{align*}
	and 
	\begin{align*}
		|j_{s,q}^{k} - p_{s,q}^{k}| \leq k+1.
	\end{align*}
	\begin{proposition}\label{monotonically increasing jump numbers}
		Let $k\in \mathbb{N}$ and $s\in F(k),$ the following are true:\\
		$(1)$ $j_{s,q}^{k} < j_{s,q+1}^{k}$ and $p_{s,q}^{k} < p_{s,q+1}^{k}$ for $q=1,\dots,k-1.$\\
		$(2)$ For any $s'\in F(k+1),$ $j_{s,k}^{k} < j_{s',1}^{k+1}$ and $p_{s,k}^{k} < p_{s',1}^{k+1}.$\\
		$(3)$ $j_{s,q}^{k},p_{s,q}^{k} \to \infty$ as $k\to \infty.$ 
	\end{proposition}
	\begin{proof}
		$(1)$ Let $q\in \{1,\dots,k-1\}$, by definition of $\{n_{1}^{k},\dots,n_{k}^{k}\}$ we have that  
		\begin{align*}
			n_{q+1}^{k} - n_{q}^k &>  2n_{q}^{k}\\
			&\geq n_{q}^{k} + n_{q}^{k}-n_{q-1}^{k}\\
			&> n_{q}^{k}-n_{q-1}^{k} + 2k + 4.
		\end{align*}
		Therefore, 
		\begin{align*}
			&\frac{n_{q}^{k}-n_{q-1}^{k}}{2} + \frac{k}{2} < \frac{n_{q}^{k}-n_{q-1}^{k}}{2} + \frac{k}{2} + 1 < 
			\frac{n_{q+1}^{k}-n_{q}^{k}}{2} - \frac{k}{2} - 1 < \frac{n_{q+1}^{k}-n_{q}^{k}}{2} - \frac{k}{2},
		\end{align*}
		and
		\begin{align*}
			&j_{s,q}^{k} \leq \frac{n_{q}^{k}-n_{q-1}^{k}}{2} + \frac{k}{2} < \frac{n_{q+1}^{k}-n_{q}^{k}}{2} - \frac{k}{2}\leq j_{s,q+1}^{k},\\
			&p_{s,q}^{k} \leq \frac{n_{q}^{k}-n_{q-1}^{k}}{2} + \frac{k}{2}  + 1 < \frac{n_{q+1}^{k}-n_{q}^{k}}{2} -  \frac{k}{2} - 1\leq p_{s,q+1}^{k}.
		\end{align*}
		Thus, $j_{s,q}^{k} < j_{s,q+1}^{k}$ and $p_{s,q}^{k} < p_{s,q+1}^{k}$.\\
		$(2)$ By definition of $n_{1}^{k+1}$ we have that 
		\begin{align*}
			n_{1}^{k+1} &> 100n_{k}^{k}\\
			&\geq99n_{k}^{k} + n_{k}^{k} - n_{k-1}^{k}\\
			&>n_{k}^{k} - n_{k-1}^{k} + (k+3) + k+2
		\end{align*}
		and 
		\begin{align*}
			\frac{n_{k}^{k}-n_{k-1}^{k}}{2} + \frac{k}{2} + 1 < \frac{n_{1}^{k+1}}{2} - \frac{k+1}{2} - 1.
		\end{align*}
		By the above inequality, it follows that $j_{s,k}^{k} < j_{s',1}^{k+1}$ and $p_{s,k}^{k} < p_{s',1}^{k+1}.$\\
		$(3)$ Since
		\begin{align*}
			&\frac{n_1^k}{2}-\frac{k}{2}\leq j_{s,1}^{k}\leq \frac{n_1^k}{2} + \frac{k}{2},\\
			&\frac{n_1^k}{2}-\frac{k}{2}-1\leq p_{s,1}^{k}\leq \frac{n_1^k}{2} + \frac{k}{2} + 1,
		\end{align*}
		we have by the definition of the $n_1^k$'s that $j_{s,1}^{k},p_{s,1}^{k}\to \infty$ as $k\to \infty.$ By $(1)$ it follows that $j_{s,q}^{k},p_{s,q}^{k}\to \infty$ as $k\to \infty$ for $q=1,\dots,k.$ 
	\end{proof}	
	\noindent For $k\in \mathbb{N}$, $s\in F(k)$, and $q=1,\dots,k$ we will call the numbers $j_{s,q}^{k}, p_{s,q}^{k}$ \textbf{jump numbers}. \\
	\indent For each $k\in \mathbb{N}$ and $s\in F(k),$ choose a finite collection of distinct points in $\mathbb{R}^2$, $Y_s = \left\{y_{s,0},\dots,y_{s,n_{k}^{k}}\right\}$ such that for $q=1,\dots,k$ 
	\begin{align*}
		&y_{s,n_{q-1}^{k}+i} \in U^{k}(a_{s(q-1)+i}), \text{ for } i=0,\dots,j_{s,q}^{k}-1\\
		&y_{s,n_{q-1}^{k}+j_{s,q}^{k} + i}\in U^{k}(a_{-p_{s,q}^{k}+i}), \text{ for } i=0,\dots,n_{q}^{k}-n_{q-1}^{k}-j_{s,q}^{k}.
	\end{align*}
	Since $p_{s,q}^k = n_{q}^{k}-n_{q-1}^k-j_{s,q}^{k}-s(q),$ we see that $n_{q}^k - n_{q-1}^{k}-j_{s,q}^{k} - p_{s,q}^{k} = s(q)$ and $y_{s,n_{q}^{k}}\in U^{k}(a_{s(q)}).$ \\
	For each $s\in F(k)$ we choose a sequence of points $\{x_{s,i}\}_{i\in\mathbb{Z}}$ such that 
	\begin{align*}
		&1.) \text{ }x_{s,i} = y_{s,i} \text{ for } 0\leq i\leq n_{k}^{k}, \\
		&2.) \text{ }x_{s,n_{k}^{k} + r} \in U^{k+r}(a_{s(k) + r}) \text{ and } x_{s,-r}\in U^{k+r}(a_{s(0)-r}), \text{ for } r\geq1,\\
		&3.) \text{ } x_{s,i}\neq x_{s,j} \text{ if } i\neq j.
	\end{align*}
	As $i\to \infty$, both $x_{s,i}$ and $x_{s,-i}$ approach $a_{\infty}.$ Let $X_{s} = \{x_{s,i}: i\in \mathbb{Z}\}\cup \{a_{\infty}\},$ then $X_{s}$ is homeomorphic to $A$ and $\lim_{k\to \infty}H_{d}\left(\bigcup_{s\in F(k)}X_{s},A\right) = 0$ as $$H_d\left(\bigcup_{s\in F(k)}X_{s},A\right) \leq 1/k,$$ where $H_{d}$ denotes the Hausdorff distance. This gives us that $A$ is the set of accumulation points of the set $\bigcup_{k=1}^{\infty}\bigcup_{s\in F(k)}X_{s}.$ We may assume $X_{s} \cap A = \{a_{\infty}\}$ and $X_{s_1}\cap X_{s_2}=\{a_{\infty}\}$ for distinct $s_1,s_2\in \bigcup_{k=1}^{\infty}F(k).$ Let $X_1 = \overline{\bigcup_{k=1}^{\infty}\bigcup_{s\in F(k)}X_{s}}$ and define $T: X_1\to X_1$ to be $T_{1} = T\rvert_{A}$ and $T(x_{s,j}) = x_{s,j+1}$ for $s\in \bigcup_{k=1}^\infty F(k).$ Since our jump numbers go to infinity, we see that $T$ is a homeomorphism of our space.  Tan, Ye, and Zhang proved the following 
	\begin{theorem}[Tan, Ye, and Zhang, Lemma 6.2 \cite{TAN_Ye_Zhang_set_of_sequence_entropes}]\label{t.d.s. with log(2) maximal pattern entropy}
		The t.d.s. $(X_1,\mathbb{Z},T)$ has $h_{top}^{*}(X_1,\mathbb{Z}) = \log(2)$ and $IN_{2}^{e}(X_1,\mathbb{Z},T) = \{(a_i,a_{i+1}),(a_{i+1},a_i): i\in \mathbb{Z})\}.$  
	\end{theorem}
	We'll show that $h_{top}^{*}\left((X_1)_{\mathbb{Z}}^{G},G\right) = \log(2)$. Before doing so, we'll prove a series of technical lemmas that will be important to us. The first lemma we'll prove is a generalized version of Lemma 5.1 found in \cite{Snhoa_Ye_Zhang_topological_sequence_entropy}. For their variation of the construction, Snoha, Ye, and Zhang proved the case when $j = 0$ and $k = 1.$ 
	\begin{lemma}\label{lemma 1}
		Let $k\in \mathbb{N}$ and $s\in F(k).$ If $j\in \mathbb{Z}$ and $|j| \leq \min\left\{j_{s,1}^{k}-1, p_{s,1}^{k}-1\right\}$ then 
		\begin{align*}
			X_{s} \cap U^{k}\left(a_j\right) = \left\{x_{s,n_{i}^{k}-s(i) + j}: i=0,\dots,k\right\}.
		\end{align*}
	\end{lemma}
	\begin{proof}
		By our construction of $\{x_{s,i}\}_{i\in \mathbb{Z}}$ we have that $x_{s,n_{q}^{k}}\in U^{k}(a_{s(q)})$. This implies that $$X_{s}\cap U^{k}(a_0) = \left\{x_{s,n_{q}^{k}-s(q)}:q=0,\dots,k\right\},$$ and 
		\begin{align*}
			X_{s}\cap U^{k}(a_1) = \left\{x_{s,n_{q}^{k}-s(q) + 1}:q=0,\dots,k\right\}.
		\end{align*}
		Assume that $j > 1$ and $j \leq \min\{j_{s,1}^{k}-1,p_{s,1}^{k}-1\}$, then $j \leq j_{s,1}^{k} -1$ and by Proposition \ref{monotonically increasing jump numbers}, $j < j_{s,q}^{k}$ for $q=1,\dots,k.$ By the construction of $\{x_{s,i}\}_{i\in \mathbb{Z}}$, we have that 
		\begin{align*}
			&1.) \text{ }x_{s,i} = y_{s,i}, \text{ for }0\leq i\leq n_{k}^{k}\\
			&2.) \text{ }x_{s,n_{k}^{k} + r} \in U^{k+r}\left(a_{s(k) + r}\right) \text{ and } x_{s,-r}\in U^{k+r}\left(a_{s(0)-r}\right),\text{ for } r\geq 1. 
		\end{align*}
		As $j > 1,$ it is clear that $x_{s,-r}\notin U^{k}(a_{j})$ for $r\geq 1.$ Let $q\in \{1,\dots,k\},$ then $x_{s,n_{q-1}^{k}+r}\in U^{k}(a_{s(q-1)+r})$ for $0\leq r < j_{s,q}^{k}.$ Since $1 < j\leq j_{s,q}^{k}-1$, we have that $0 < j-s(q-1) < j_{s,q}^{k}$. Whenever $r = j-s(q-1)$ we have that $x_{s,n_{q-1}^{k}+j-s(q-1)}\in U^{k}(a_{j}).$ For $0 \leq r < j-s(q-1)$ or $j-s(q-1) <  r < j_{s,q}^{k}$, it can be easily seen that $x_{s,n_{q-1}^{k}+r}\notin U^{k}(a_{j}).$ For $i\in \{n_{q-1}^{k}+j_{s,q}^{k},\dots, n_{q}^{k}\}$, we have that $x_{s,i} = y_{s,n_{q-1}^{k}+j_{s,q}^{k} + r}$ for some $r \in \{0,\dots, n_{q}^{k}-n_{q-1}^{k}-j_{s,q}^{k}\}.$ As $y_{s,n_{q-1}^{k}+j_{s,q}^{k}+r}\in U^{k}(a_{-p_{s,q}^{k} + r})$, and $r - p_{s,q}^{k} \leq 1$ for all $r$ we see that $x_{s,i}\notin U^{k}(a_j).$ If $i = n_{k}^{k}+j-s(k),$ as $j > 1$ we know that $j-s(k) >0$ and $x_{s,n_{k}^{k}+j-s(k)}\in U^{k+j-s(k)}(a_{j})\subseteq U^{k}(a_j).$ For $0 < r < j-s(k)$ or $j-s(k) < r$, it can be easily seen that $x_{s,n_{k}^{k}+r}\notin U^{k}(a_j).$ As this exhausted all possible options for $i,$ we see that in fact 
		$$\left\{x_{s,n_{q}^{k}-s(q)+j}:q=0,\dots,k\right\} = X_{s}\cap U^{k}(a_j).$$ 
		If $j< 0,$ the proof is the same but instead of using the $j_{s,q}^{k}$ numbers, the $p_{s,q}^{k}$ numbers are used instead.
	\end{proof}
	Before proving our next result, let's define two terms. The first is iterative distance. Let $k\in \mathbb{N}$ and $s\in F(k).$ For $x,y\in X_{s}\backslash \{a_{\infty}\}$ we know that $x = x_{s,i}$,$y=x_{s,j}$ for some $i,j\in\mathbb{Z}.$ The \textbf{iterative distance} between $x,y$ is denoted by $I(x,y) = |i-j|.$ If $j > i,$ we say that $y>x.$ The next definition we will need is \textbf{similar positions}. Let $k > 0, s_1,s_2\in F(k).$ We say that $x\in X_{s_1}\cap U^{1}(a_0)$ and $y\in X_{s_2}\cap U^{1}(a_0)$ are in similar positions if for some $i\in \{0,\dots,k\}$
	\begin{align*}
		x &= x_{s_1,n_{i}^{k}-s_1(i)},\\
		y&= x_{s_2,n_{i}^{k}-s_2(i)}.
	\end{align*}
	In \cite{Snhoa_Ye_Zhang_topological_sequence_entropy} Snoha, Ye, and Zhang introduced the idea of similar positions for their variation of the construction. Their construction is similar in spirit, but different in some ways. For this reason, we'll prove a variation of Lemma 5.4(4) in \cite{Snhoa_Ye_Zhang_topological_sequence_entropy} for our context.
	Before proving our next result, let's define two terms. The first is iterative distance. Let $k\in \mathbb{N}$ and $s\in F(k).$ For $x,y\in X_{s}\backslash \{a_\infty\}$ we know that $x = x_{s,i}$, $y=x_{s,j}$ for some $i,j\in\mathbb{Z}.$ The \textbf{iterative distance} between $x,y$ is denoted by $I(x,y) = |i-j|.$ If $j > i,$ we say that $y>x.$ The next definition we will need is \textbf{similar positions}. Let $k \in \mathbb{N}$ and $s_1,s_2\in F(k).$ We say that $x\in X_{s_1}\cap U^{1}(a_0)$ and $y\in X_{s_2}\cap U^{1}(a_0)$ are in similar positions if for some $i\in \{0,\dots,k\}$
	\begin{align*}
		x &= x_{s_1,n_{i}^{k}-s_1(i)},\\
		y&= x_{s_2,n_{i}^{k}-s_2(i)}.
	\end{align*}
	In \cite{Snhoa_Ye_Zhang_topological_sequence_entropy} Snoha, Ye, and Zhang introduced the idea of similar positions for their variation of the construction. Their construction is similar in spirit, but different in some ways. For this reason, we'll prove the following variation of Lemma 5.4(4) in \cite{Snhoa_Ye_Zhang_topological_sequence_entropy} for our context.
	\begin{lemma}\label{lemma 2}
		Let $k,k' \in \mathbb{N}$ and $s\in F(k),s'\in F(k').$ If $x,y\in X_{s}\cap U^{1}(a_0)$ and $z,w\in X_{s'}\cap U^{1}(a_0)$ such that $x<y,z < w$, and $I(x,y) = I(z,w),$ then $k = k'$, $x$ is in a similar position as $z$, and $y$ is in a similar position as $w.$  
	\end{lemma}
	\begin{proof}
		By Lemma \ref{lemma 1} since $x < y$ and $z< w$ for some $0\leq i < j \leq k$ and $0\leq r < p \leq k'$
		\begin{align*}
			x &= x_{s,n_i^{k}-s(i)},\\
			y&= x_{s,n_j^{k}-s(j)},\\
			z&= x_{s',n_r^{k'}-s'(r)},\\
			w &= x_{s',n_p^{k'}-s'(p)}.  
		\end{align*}
		As $i < j$ and $r < p,$ we have that $0\leq n_{i}^{k}< n_{j}^{k}\leq n_{k}^{k}$ and $0 \leq n_{r}^{k'} < n_{p}^{k'}\leq n_{k'}^{k'}.$ The iterative distances between $x,y$ and $z,w$ are equal, so we have that 
		\begin{align*}
			n_{p}^{k'} - s'(p) - n_{r}^{k'} + s'(r) = n_{j}^{k}-s(j) - n_{i}^{k} + s(i). 
		\end{align*}
		Assume $k'  > k$ and $p > 1.$ By our construction $n_{p}^{k'} > 100\times n_{p-1}^{k'}$, $n_{p-1}^{k'}\geq n_{r}^{k'}$ and $n_{p-1}^{k'} > n_{j}^{k}$ since $k' > k.$ Thus, 
		\begin{align*}
			n_{p}^{k'} - s'(p) - n_{r}^{k'} + s'(r) &> 99\times n_{p-1}^{k'} - s'(p) + s'(r)\\
			&> 98\times n_{p-1}^{k'} + n_{j}^{k} - s'(p) + s'(r)\\
			&> n_{j}^{k} - s(j) - n_{i}^{k} + s(i),
		\end{align*}
		and we arrive to a contradiction. If $p = 1$ and $k' > k,$ then $r=0$ and by our construction $n_{1}^{k'} > 100\times n_{k}^{k}$ and $n_{k}^{k}\geq n_{j}^{k}.$ Thus, 
		\begin{align*}
			n_p^{k'} - s'(p) - n_{r}^{k'} + s'(r) &= n_{1}^{k'} - s'(1) + s'(0)\\
			&> 100\times n_{k}^{k} - s'(1) + s'(0)\\
			&\geq 99\times n_{k}^{k} + n_{j}^{k} - s'(1) + s'(0)\\
			&> n_{j}^{k} - s(j) - n_{i}^{k} + s(i),
		\end{align*}
		and we arrive to a contradiction. If $k' < k$, we will arrive to a similar contradiction; therefore, $k = k'.$ \\
		\indent Next, we show that $y$ and $w$ must be in similar positions. Assume $p > j,$ then $n_{p}^{k} > n_{j}^{k}$. By our construction, $n_{p}^{k} > 100\times n_{p-1}^{k}$ and $n_{p-1}^{k}\geq \max\{n_{j}^{k},n_{r}^{k}\}.$ Therefore, 
		\begin{align*}
			n_{p}^{k} - s'(p) - n_{r}^{k} + s'(r) &> 99\times n_{p-1}^{k} - s'(p) + s'(r)\\
			&\geq 98\times n_{p-1}^{k} + n_{j}^{k} - s'(p) + s'(r)\\
			&>n_{j}^{k} - s(j) - n_{i}^{k}+s(i).
		\end{align*}
		Since our iterative distance between $x,y$ is the same as the iterative distance between $z,w$ we result in a contradiction if $p > j.$ Similarly, if $j > p$ we result in the same contradiction. Therefore, $y$ and $w$ must be in similar positions, and $w = x_{s',n_{j}^{k}-s'(j)}.$ To show that $x,z$ are in similar positions assume for contradiction that $i > r.$ As $I(z,w) = I(x,y)$ we have that 
		\begin{align*}
			n_{j}^{k}-n_{r}^{k} -s'(j)+s'(r) = n_{j}^{k} - n_{i}^{k}-s(j)+s(i).
		\end{align*} 
		This implies that 
		\begin{align*}
			n_{i}^{k} - n_{r}^{k} = (s'(j) -s'(r)) - (s(j)-s(i)) \in \{1,2\}
		\end{align*}
		since $n_{i}^{k} > n_{r}^{k}.$ This leads to a contradiction, since $n_{i}^{k}$ is significantly larger than $n_{r}^k.$ We result in a similar contradiction if one assumes that $r > i;$ therefore, $i = r$ and $x,z$ are in similar positions. 
	\end{proof}
	\begin{lemma}\label{lemma 3}
		Let $k,k'\in \mathbb{N}$ such that $k' > k$ and $j\in \mathbb{Z}.$ If $s\in F(k'),$ then 
		\begin{align*}
			X_{s}\cap U^{k}(a_j) = X_{s}\cap U^{k'}(a_j).
		\end{align*}
	\end{lemma}
	\begin{proof}
		As $k' > k$, we see immediately that $X_{s}\cap U^{k'}(a_j)\subseteq X_{s}\cap U^{k}(a_j).$ For the reverse direction, let $x\in X_{s}\cap U^{k}(a_j)$ then $x = x_{s,i}$ for some $i\in \mathbb{Z}.$ By the selection of $\{x_{s,i}\}_{i\in \mathbb{Z}},$ since $s\in F(k')$ we have that $x_{s,i}\in U^{k'}(a_r)$ for some $r\in \mathbb{Z}.$ As $U^{k'}(a_r)\cap U^{k}(a_j) = \emptyset$ unless $j=r,$ we see that $x_{s,i}\in U^{k'}(a_j)$, $X_{s}\cap U^{k}(a_j)\subseteq X_{s}\cap U^{k'}(a_j),$ and our equality follows.  
	\end{proof}
	\begin{lemma}\label{the key lemma}
		Let $i,j\in \mathbb{Z},$ then $(a_i,a_j)\notin IN_{1}\left((X_1)_{\mathbb{Z}}^{G},G,T^{\mathbb{Z}\backslash G}\right).$ 
	\end{lemma}
	\begin{proof}
		Assume for contradiction there are $i,j\in \mathbb{Z}$ such that $$(a_i,a_j)\in IN_{1}\left((X_1)_{\mathbb{Z}}^{G},G,T^{\mathbb{Z}\backslash G}\right).$$ By Proposition \ref{finite index subgroup maximal pattern entropy IN-tuple lemma} $$IN_{1}\left((X_1)_{\mathbb{Z}}^{G},G,T^{\mathbb{Z}\backslash G}\right) = IN_{1}\left((X_1)_{\mathbb{Z}}^{G},\mathbb{Z},T^{\mathbb{Z}\backslash G}\rvert_{\mathbb{Z}}\right).$$ Therefore, we can merely consider independence sets contained within $\mathbb{Z}.$ By Theorem \ref{properties of IN-tuples}(3), $$(a_0,a_{j+i})=T^{\mathbb{Z}\backslash G}_{(-i,0)}(a_{i},a_{j}) \in IN_1\left((X_1)_{\mathbb{Z}}^{G},\mathbb{Z},T^{\mathbb{Z}\backslash G}\rvert_{\mathbb{Z}}\right).$$ 
		Therefore, we can assume without loss of generality that $i=0$ and $j\in \mathbb{Z}.$ By Proposition \ref{monotonically increasing jump numbers} there exists a $k\in \mathbb{N}$ such that $|j|\leq \min\{j_{s,1}^{k}-1,p_{s,1}^{k}-1\}$ for all $s\in F(k).$ The open set $$U = (U^{k}(a_0)\cap X_1)\times (U^{k}(a_j)\cap X_1)$$ is an open neighborhood of $(a_0,a_j).$ As $(a_0,a_j)\in IN_{1}\left((X_1)_{\mathbb{Z}}^{G},\mathbb{Z},T^{\mathbb{Z}\backslash G}\rvert_{\mathbb{Z}}\right)$ there exists an independence set of $U$,  $\{\ell_1,\dots,\ell_n\},$ where $\ell_{i} < \ell_{i+1}$ for $i=1,\dots,n-1$ and $n > k+1.$ By independence, there are $x,y\in X_1$ such that 
		\begin{align*}
			&T^{\ell_1}(x),\dots,T^{\ell_n}(x)\in U^{k}(a_0)\cap X_1,\\
			&T^{-\ell_1}(y),\dots,T^{-\ell_n}(y)\in U^{k}(a_j)\cap X_1.
		\end{align*}
		The orbits of both $x$ and $y$ return to $U^{k}(a_0)\cap X_1,U^{k}(a_j)\cap X_1$ at least $n$ times, respectively; therefore there are natural numbers $k',k''\geq n-1$, $s'\in F(k'),$ and $s''\in F(k'')$ such that $\text{orb}(x) = X_{s'}\backslash \{a_\infty\}$ and $\text{orb}(y) = X_{s''}\backslash \{a_{\infty}\}.$ Since $k',k'' > k,$ by Lemma \ref{lemma 3}
		\begin{align*}
			&T^{\ell_1}(x),\dots,T^{\ell_n}(x)\in U^{k'}(a_0)\cap X_{s'},\\
			&T^{-\ell_1}(y),\dots,T^{-\ell_n}(y)\in U^{k''}(a_j)\cap X_{s''}.
		\end{align*}
		Since $k'' > k$ by Proposition \ref{monotonically increasing jump numbers} we have that $|j| \leq \min\{j_{s'',1}^{k''}-1,p_{s'',1}^{k''}-1\}$. Therefore, by Lemma \ref{lemma 1} it follows that $T^{-\ell_{i}}(y') \in U^{k''}(a_0)\cap X_{s''}$ for $i=1,\dots, n$ where $y' = T^{-j}(y).$ Thus, we have $x,y'\in X_1$ such that 
		\begin{align*}
			&T^{\ell_1}(x),\dots,T^{\ell_n}(x)\in U^{k'}(a_0)\cap X_{s'}\subseteq U^{1}(a_0)\cap X_{s'},\\
			&T^{-\ell_1}(y'),\dots,T^{-\ell_n}(y')\in U^{k''}(a_0)\cap X_{s''}\subseteq U^{1}(a_0)\cap X_{s''}.
		\end{align*}
		In terms of iterative distance, we have that $T^{\ell_1}(x) < T^{\ell_2}(x)$, $T^{-\ell_2}(y')<T^{-\ell_1}(y'),$ and $$I(T^{\ell_1}(x),T^{\ell_2}(x)) = \ell_2 - \ell_1 = I(T^{-\ell_2}(y'),T^{-\ell_1}(y')).$$ It follows from Lemma \ref{lemma 2} that $k'=k''$ and $T^{\ell_2}(x)$, $T^{-\ell_1}(y')$ are in similar positions. Furthermore, $T^{\ell_1}(x) < T^{\ell_3}(x)$, $T^{-\ell_3}(y')<T^{-\ell_1}(y'),$ and  $$I(T^{\ell_1}(x),T^{\ell_3}(x)) = \ell_3 - \ell_1 =  I(T^{-\ell_3}(y'),T^{-\ell_1}(y')).$$ Therefore, $T^{\ell_3}(x)$, $T^{-\ell_1}(y')$ are in similar positions as well by another application of Lemma \ref{lemma 2}. As $\ell_2 < \ell_3$ this leads to a contradiction. Thus, $U$ cannot have an independence set of length $n > k+1$ and $(a_0,a_j)\notin  IN_{1}\left((X_1)_{\mathbb{Z}}^{G},G,T^{\mathbb{Z}\backslash G}\right).$
	\end{proof}
	\begin{theorem}\label{co-induction counter-example calculation}
		$h_{top}^{*}\left((X_1)_{\mathbb{Z}}^{G},G,T^{\mathbb{Z}\backslash G}\right) = \log(2).$ 
	\end{theorem}
	\begin{proof}
		By Theorem \ref{t.d.s. with log(2) maximal pattern entropy} and the discussion at the beginning of Section \ref{maximal pattern entropy finite index subsection} $$\log(2) = h_{top}^{*}(X_1,\mathbb{Z})\leq h_{top}^{*}\left((X_1)_{\mathbb{Z}}^{G},G\right).$$ To prove the reverse direction we will show that $IN_{3}^{e}\left((X_1)_{\mathbb{Z}}^{G},G,T^{\mathbb{Z}\backslash G}\right) = \emptyset.$ Assume for contradiction that $(f_1,f_2,f_3)\in IN_{3}^{e}\left((X_1)_{\mathbb{Z}}^{G},G,T^{\mathbb{Z}\backslash G}\right).$ By Lemma \ref{IN-tuples in the finite index case}, it follows that $$(f_1(0),f_2(0),f_3(0)),(f_1(x),f_2(x),f_3(x))\in IN_{3}(X_1,\mathbb{Z}).$$ By Theorem \ref{t.d.s. with log(2) maximal pattern entropy} it follows that $IN_{3}^{e}(X_1,\mathbb{Z}) = \emptyset.$  Therefore, $f_1(0),f_2(0),f_3(0)$ cannot be all distinct nor can they all be equal. For if they were all equal then $f_1(x),f_2(x),f_3(x)$ would have to be all distinct for $f_1,f_2,f_3$ to be distinct points in $(X_1)_{\mathbb{Z}}^{G}.$ Thus, we conclude that $\{f_1(0),f_2(0),f_{3}(0)\} = \{x_1,x_2\}$ and $\{f_1(x),f_2(x),f_3(x)\} = \{y_1,y_2\}$ where $(x_1,x_2),(y_1,y_2)\in IN_{2}^{e}(X_1,\mathbb{Z}).$ By Theorem \ref{t.d.s. with log(2) maximal pattern entropy} we have that $\{x_1,x_2\} = \{a_i,a_{i+1}\}$ and $\{y_1,y_2\}=\{a_j,a_{j+1}\}$ for some $i,j\in \mathbb{Z}.$ Without loss of generality assume that $x_1 = a_i,x_2=a_{i+1},y_1 = a_{j},$ and $y_{2} = a_{j+1}.$ This implies that $$f_1,f_2,f_3\in \{(a_{i},a_{j}),(a_{i},a_{j+1}),(a_{i+1},a_{j}),(a_{i+1},a_{j+1})\}.$$ However, by Lemma \ref{the key lemma} it follows that $f_1,f_2,f_3\notin IN_{1}\left((X_1)_{\mathbb{Z}}^{G},G,T^{\mathbb{Z}\backslash G}\right)$ and $(f_1,f_2,f_3)$ cannot be an $IN$-tuple. This leads us to a contradiction; therefore, $IN_{3}^{e}\left((X_1)_{\mathbb{Z}}^{G},G,T^{\mathbb{Z}\backslash G}\right) = \emptyset$. By Theorem \ref{local maximal pattern entropy} it follows that $h_{top}^{*}\left((X_1)_{\mathbb{Z}}^{G},G\right) = \log(2)$ as desired. 
	\end{proof}		
	\begin{remark}
		\normalfont In our proof, we used the fact that $$\log(2) = h_{top}^{*}(X_1,\mathbb{Z})\leq h_{top}^{*}\left((X_1)_{\mathbb{Z}}^{G},G\right).$$ This implies that $IN_{2}^{e}\left((X_1)_{\mathbb{Z}}^{G},G,T^{\mathbb{Z}\backslash G}\right)\neq \emptyset,$ but it does not gives us any information about the intrinsic IN-pair that exists. As $a_{\infty}$ is a fixed-point of $T: X_1\to X_1,$ then $\mathbb{Z}$ is an infinite independence set for any open neighborhood about $a_{\infty}$ in $X_1.$ Using this fact, it is easy to see that if we define $f_1 = (a_0,a_{\infty})$ and $f_2 = (a_1,a_{\infty})$ that $(f_1,f_2)\in IN_{2}^{e}\left((X_1)_{\mathbb{Z}}^{G},G,T^{\mathbb{Z}\backslash G}\right).$ 
	\end{remark}
	\subsection{Maximal Pattern Entropy of co-Induced Systems: $[G:H] = +\infty$}\label{maximal pattern entropy infinite index subsection}
	\indent In this section, we will compute the maximal pattern entropy of co-induced systems when $[G:H] = +\infty$. We will show if $[G:H]=+\infty,$ then $IN_k\left(X_{H}^{G},G\right) = \left(X_{H}^{G}\right)^{k}$ for all $k\in \mathbb{N}$ and $h_{top}^{*}(X_{H}^{G},G) = +\infty.$ We will start this discussion by stating and proving some results. 
	\begin{proposition}[Neumann, Proposition 4.2 \cite{Neumann_group_lemma}]\label{BH Neumman Result}
		Let the group $G$ be the union of finitely many, let us say $n$, cosets of subgroups $C_1,C_2,\dots,C_n:$ 
		\begin{align*}
			G=\bigcup_{i=1}^{n}C_ig_i,
		\end{align*}
		then at least one subgroup $C_i$ has finite index in $G$.
	\end{proposition}
	We have the following consequence of Proposition \ref{BH Neumman Result} that will be important in proving Theorem \ref{infinite index subgroup co-induction has infinite maximal pattern entropy}. 
	\begin{lemma}\label{the key to proving my theorem involving infinite index subgroup actions}
		Let $H\leq G$, $s: H\backslash G \to G$ a fixed section of $H\backslash G$, $K\in \mathscr{F}(H\backslash G),$ and $F\in \mathscr{F}(G).$ If $[G:H] = +\infty$ then 
		\begin{align*}
			G\backslash\left(\bigcup_{g\in F}\bigcup_{\theta\in K}\bigcup_{\phi\in K}s(\theta)^{-1}Hs(\phi)g\right) \neq \emptyset
		\end{align*}
		and contains infinitely many elements. 
	\end{lemma}
	\begin{proof}
		We re-write the union as 
		\begin{align*}
			\bigcup_{g\in F}\bigcup_{\theta\in K}\bigcup_{\phi\in K}\left[s(\theta)^{-1}Hs(\theta)\right]s(\theta)^{-1}s(\phi)g. 
		\end{align*}
		For each $\theta\in K$ the subgroup $s(\theta)^{-1}Hs(\theta)$ has infinite index in $G$ because $H$ is an infinite index subgroup of $G$. Proposition \ref{BH Neumman Result} tells us if the above union equals $G$ then one of the subgroups must be of finite index. This would lead to a contradiction; therefore, 
		\begin{align*}
			G\backslash\left(\bigcup_{g\in F}\bigcup_{\theta\in K}\bigcup_{\phi\in K}s(\theta)^{-1}Hs(\phi)g\right) \neq \emptyset.
		\end{align*}
		Moreover, it
		must contain infinitely many elements of $G$. If not, we could cover $G$ with finitely many right cosets of subgroups with infinite index, which by Proposition \ref{BH Neumman Result} is impossible.  		
	\end{proof}
	\begin{theorem}\label{infinite index subgroup co-induction has infinite maximal pattern entropy}
		Let $(X,H,\alpha)$ be a t.d.s. where $X$ contains at least two distinct points and $(X_{H}^{G},G,\alpha^{H\backslash G})$ be the co-induced system. If $[G:H] = +\infty$, then $$IN_{k}\left(X_{H}^{G},G\right) = \left(X_{H}^{G}\right)^{k}$$ for all $k\in \mathbb{N}$ and $h_{top}^{*}(X_{H}^{G},G) = +\infty.$ 
	\end{theorem} 
	\begin{proof}
		Let $k\in \mathbb{N}$, $(f_1,\dots,f_k)\in \left(X_{H}^{G}\right)^{k}$ and $U_1\times \ldots \times U_k$ be a product neighborhood of $(f_1,\dots,f_k)$. We can assume without loss of generality there exists $K\in \mathscr{F}(H\backslash G)$ such that $K=\{Ha_1,\dots,Ha_r\}$ and 
		\begin{align*}
			U_i = \bigcap_{j=1}^{r}\pi_{Ha_j}^{-1}(V_{j,i}),
		\end{align*}
		where $V_{j,i}$ are non-empty, open subsets of $X$. We will assume that $s(Ha_i) = a_i$ for $1\leq i\leq r.$ Let $g_1 = e$ and $F_1 = \{e\}$, then $F_1$ is clearly an independence set for $(U_1,\dots,U_k).$ By Lemma \ref{the key to proving my theorem involving infinite index subgroup actions}
		\begin{align*}
			G\backslash \left(\bigcup_{i=1}^{r}\bigcup_{j=1}^{r}a_i^{-1}Ha_jg_1\right) \neq \emptyset,
		\end{align*}
		so we choose $g_2$ as an element of this set. This gives us that $Ha_ig_1 \neq Ha_jg_2$ for any $1\leq i,j\leq r.$ Let $F_2 = \{g_1,g_2\}$ and $\sigma\in [k]^{\{g_1,g_2\}},$ then
		\begin{align*}
			\bigcap_{p=1}^{2}\alpha^{H\backslash G}_{g_p^{-1}}\left(U_{\sigma(g_p)} \right) &= \bigcap_{p=1}^{2}\bigcap_{j=1}^{r}\alpha^{H\backslash G}_{g_{p}^{-1}}\pi_{Ha_j}^{-1}\left(V_{j,\sigma(g_p)}\right)\\
			&= \bigcap_{p=1}^{2}\bigcap_{j=1}^{r}\pi_{Ha_jg_p}^{-1}\left(\alpha_{s(Ha_jg_p)g_{p}^{-1}a_j^{-1}}V_{j,\sigma(g_p)}\right)\\
			&\neq \emptyset.
		\end{align*} 
		Therefore, $F_2$ is an independence set for $(U_1,\dots,U_k).$ We continue this construction recursively. At each step we use Lemma \ref{the key to proving my theorem involving infinite index subgroup actions} to choose \begin{align*}
			g_n\notin \bigcup_{p=1}^{n-1}\bigcup_{i=1}^{r}\bigcup_{j=1}^{r}a_i^{-1}Ha_jg_p
		\end{align*}so $Ha_ig_q\neq Ha_jg_p$ for $1\leq i,j\leq r$ and $1\leq q,p\leq n$ such that $p\neq q.$ Let $F_n = \{g_1,\dots, g_n\}$ and $\sigma\in [k]^{\{g_1,\dots,g_n\}},$ then
		\begin{align*}
			\bigcap_{p=1}^{n}\alpha^{H\backslash G}_{g_p^{-1}}\left(U_{\sigma(g_p)} \right) &= \bigcap_{p=1}^{n}\bigcap_{j=1}^{r}\alpha^{H\backslash G}_{g_{p}^{-1}}\pi_{Ha_j}^{-1}\left(V_{j,\sigma(g_p)}\right)\\
			&= \bigcap_{p=1}^{n}\bigcap_{j=1}^{r}\pi_{Ha_jg_p}^{-1}\left(\alpha_{s(Ha_jg_p)g_{p}^{-1}a_j^{-1}}V_{j,\sigma(g_p)}\right)\\
			&\neq \emptyset.
		\end{align*} 
		Therefore, $F_n$ is an independence set for $(U_1,\dots, U_k)$. We can continue this argument to make arbitrarily large, finite independence sets for $(U_1,\dots,U_k)$. We apply this argument to an arbitrary product neighborhood of $(f_1,\dots,f_k);$ therefore, $(f_1,\dots,f_k)\in IN_{k}\left(X_{H}^{G},G\right)$ and $IN_{k}\left(X_{H}^{G},G\right) = \left(X_{H}^{G}\right)^{k}.$ We have that $IN_{k}^{e}\left(X_{H}^{G},G\right) \neq \emptyset$ for all $k\in \mathbb{N};$ therefore, by Theorem \ref{local maximal pattern entropy} it follows that $h_{top}^{*}\left(X_{H}^{G},G\right) = +\infty.$
	\end{proof}
	\begin{remark}
		\normalfont In the proof of Theorem \ref{infinite index subgroup co-induction has infinite maximal pattern entropy}, we recursively construct a sequence $\{g_p\}_{p\in \mathbb{N}}$ such that for all $n\in \mathbb{N}$ the collection $\{g_1,\dots,g_n\}$ is an independence set for $(U_1,\dots,U_k).$ This gives us that $\{g_p\}_{p\in \mathbb{N}}$ is an infinite independence set for $(U_1,\dots,U_k)$ and $(f_1,\dots,f_k)$ is an $IT$-tuple of length $k$ (see Section 6 of \cite{Kerr_Li_Independence}). Therefore, for $k\in \mathbb{N}$ we have that every tuple of length $k$ is an $IT$-tuple.  
	\end{remark}
	\begin{remark}
		\normalfont When $G$ is amenable and $H \leq G,$ the topological entropy $h_{top}(X_{H}^{G},G) = h_{top}(X,H)$ by Theorem 4.2 in \cite{Dooley_Zhang_co_induction}. Using Theorem \ref{infinite index subgroup co-induction has infinite maximal pattern entropy}, we can construct co-induced systems such that $h_{top}(X_{H}^{G},G) = 0$ and $h_{top}^{*}(X_{H}^{G},G) = +\infty.$ For instance, consider the example we discussed following the proof of Theorem \ref{null iff null in finite index case}. We let $X = \{0,1\}$, $H = \mathbb{Z},$ and $\alpha$ be the trivial action of $\mathbb{Z}$ on $X.$ The group $\mathbb{Z}\times \mathbb{Z}$ is an amenable group, $h_{top}^{*}\left(\{0,1\}_{\mathbb{Z}}^{\mathbb{Z}\times\mathbb{Z}},\mathbb{Z}\times\mathbb{Z}\right) = +\infty,$ and $h_{top}\left(\{0,1\}_{\mathbb{Z}}^{\mathbb{Z}\times\mathbb{Z}},\mathbb{Z}\times\mathbb{Z}\right) = 0.$
	\end{remark}
	\section*{Acknowledgments}		
	I would like to thank my advisor, Dr. Hanfeng Li, for his guidance and insightful comments during our weekly meetings. His knowledge and wisdom have been invaluable resources throughout this project. 
	\bibliographystyle{plain}
	\bibliography{references.bib}
	\nocite{Kerr_Li_Ergodic_Theory_Book}
	\nocite{Kushnirenko_measure_sequence_entropy}
	
	\noindent\small{Department of Mathematics, SUNY at Buffalo, Buffalo, NY 14260-2900, USA}\\
	\small{\textit{Email address:} \href{mailto:dl42@buffalo.edu}{dl42@buffalo.edu}
	\end{document}	
	
	\typeout{get arXiv to do 4 passes: Label(s) may have changed. Rerun}